\documentclass[12pt,a4paper]{article} 
\usepackage{amsmath}
\usepackage{amsthm} \usepackage{amscd} \usepackage{amssymb}
\usepackage{a4wide} \usepackage{graphicx} \usepackage{multirow}

\usepackage{todonotes} \usepackage{tikz-cd}

\newtheorem{theorem}{Theorem}[section]
\newtheorem{lemma}[theorem]{Lemma}
\newtheorem{corollary}[theorem]{Corollary}
\newtheorem{proposition}[theorem]{Proposition}

\theoremstyle{definition} 
\newtheorem{example}[theorem]{Example}

\theoremstyle{remark} \newtheorem*{remark}{Remark}
 
\newtheorem*{ack}{Acknowledgments} 

\numberwithin{equation} {section} \newcounter{temp}

\makeatletter \def\square{\RIfM@\bgroup\else$\bgroup\aftergroup$\fi
\vcenter{\hrule\hbox{\vrule\@height.6em\kern.6em\vrule}\hrule}\egroup}
\makeatother

\DeclareMathOperator{\Def}{Def} \DeclareMathOperator{\depth}{depth}
\DeclareMathOperator{\Der}{Der} 
 \DeclareMathOperator{\Ext}{Ext}
 \DeclareMathOperator{\GL}{GL}
 \DeclareMathOperator{\Hom}{Hom}
 
\DeclareMathOperator{\id}{id}\DeclareMathOperator{\Image}{Im}
 \DeclareMathOperator{\Ker}{Ker}

\DeclareMathOperator{\Proj}{Proj} 
 \DeclareMathOperator{\SL}{SL}
 \DeclareMathOperator{\sort}{sort}
\DeclareMathOperator{\Spec}{Spec}
 
\DeclareMathOperator{\Sym}{Sym}

\newcommand{\baseRing}[1]{\ensuremath{\mathbb{#1}}}
 \newcommand{\ZZ}{\baseRing{Z}}
\newcommand{\Sch}{\baseRing{S}}

\newcommand{\CO}{\mathcal{O}} \newcommand{\PP}{\baseRing{P}}
 
 \newcommand{\CS}{\mathcal{S}}
\newcommand{\CR}{\mathcal{R}}
\newcommand{\CQ}{\mathcal{Q}}

\newcommand{\GG}{\baseRing{G}}
\newcommand{\OG}{\baseRing{OG}}
\newcommand{\LG}{\baseRing{LG}}
\newcommand{\g}{\mathfrak{g}}
\newcommand{\spn}{\mathfrak{sp}}
\newcommand{\so}{\mathfrak{so}}

\begin{document}

\title{Vanishing cotangent cohomology for  Pl\"ucker
  algebras}\author{Jan Arthur Christophersen \and Nathan Owen Ilten} 
\date{\today}
\maketitle

\begin{abstract} We  use representation theory and Bott's theorem to
  show vanishing of higher cotangent cohomology modules for the homogeneous
  coordinate ring of Grassmannians in the Pl\"ucker embedding. As a
  biproduct we answer a question of Wahl about the cohomology of the
  square of the ideal sheaf for the case of Pl\"ucker relations. We obtain slightly weaker vanishing results for the cotangent cohomology of the coordinate rings of isotropic Grassmannians. 
\end{abstract}

\section*{Introduction} 
Fix a field $k$ of characteristic
zero. If $\GG = \GG(r,n)$ is the
Grassmannian of $r$-planes in an $n$-dimensional vector space over
$k$, let $A$ be the corresponding  Pl\"ucker algebra, i.e.\ $\GG =
\Proj A$ in the Pl\"ucker embedding. Set $d = n(n-r)+1$ to be the Krull
dimension of $A$. 
Let $T^i_A = T^i(A/k;A)$ denote the cotangent
cohomology modules of $A$. We show that if $\GG \ne \GG(2,4)$ then $T^i_A = 0$
  for all $1\le i \le d-1$. Moreover, $T^{d}_A = 0$ if and only if $\GG= \GG(2,n)$
or $\GG = \GG(n-2,n)$. We give an example,
$\GG(2,6)$, where $T^{d+1}_A \ne 0$. 

The case $T^d_A$ is of special interest since it is the vector space
dual of $H^0_\mathfrak{m}(\Omega_A)$, where $\mathfrak{m}$ is the
irrelevant maximal ideal and $\Omega_A$ is the module of K{\"a}hler
differentials. The degree 2 part of this is  isomorphic to the 
kernel of the Gaussian map $$\bigwedge^2 H^0(\GG,\CO_\GG(1)) \to
H^0(\GG, \Omega^1_\GG 
\otimes \CO_\GG(1)^2)$$ where $\CO_\GG(1)$ is for the  Pl\"ucker
embedding. We show that the graded pieces
$H^0_{\mathfrak{m}}(\Omega_A)_m = 0$ for $m \ne 2$, which is an affirmative
answer to a question by  Jonathan Wahl in the case $G/P$ is
  a Grassmannian. See Theorem \ref{h0mo} and the following remark.  

Since $A$ is Cohen-Macaulay and $\Spec A$ has one singular point at
$\mathfrak{m}$ we have isomorphisms $T^i_A \simeq \Ext^i_A(\Omega_A,
A)$. Because of the isolated singularity we furthermore get  $\Ext^i_A(\Omega_A,
A) \simeq H^{i+1}_\mathfrak{m}(\Der_k(A))$ for $1 \le i \le
d-2$. In general, the vanishing of these local cohomology modules in the case $X =
\Proj A$  is smooth is related to cohomology  vanishing for  twists of
$\CO_X$ and $\Theta_X$. Thus vanishing of $T^i_A$ in the range $1 \le i \le
d-2$ may be shown by proving vanishing of $H^i(X, \CO_X(m))$ and
$H^i(X, \Theta_X(m))$, a result originally shown by Svanes. For our
$\GG(r,n)$ we use Bott's theorem and an argument involving the Atiyah
extension to show  $H^{i+1}_\mathfrak{m}(\Der_k(A))= 0$ for $1 \le i \le
d-2$. See Section \ref{abouthder}. 

For the remaining two cases, by local duality we have $T^{d-i}_A
\simeq H^i_\mathfrak{m}(\Omega_A)^\ast$ for $i=0,1$. Here $M^\ast$
denotes the $k$-dual. If $G = \SL_r$ and $S =k[x_{ij}: 1 \le i \le n,
1 \le j \le r]$ is the ring of functions on the vector space of $n \times r$
matrices, then $A = S^G$. In Section \ref{genHm} we consider the general
situation where $S$ is a finitely generated standard graded
$k$-algebra with the action of a 
linearly reductive group $G$ respecting the grading. 

We must assume  that $\Spec
S^G$ has an isolated 
singularity at the irrelevant maximal ideal $\mathfrak{m} \subset
S^G$ and that both $\depth_{\mathfrak{m}S}
S\ge 2$ and $\depth_{\mathfrak{m}S}
\Omega_S\ge 2$. Under these conditions we exhibit a four term complex
of free $S[G]$-modules, which after taking invariants computes
$H^i_\mathfrak{m}(\Omega_{S^G})$ for $i=0,1$. This allows us to use
representation theory to compute the local cohomology. We do this for
our case using the combinatorics of Schur functors in Section
\ref{abouthm}. 

In the case of isotropic Grassmannians, we also understand enough about the tangent sheaf to apply Bott's theorem to get results similar to above, see Section \ref{sec:iso}.
Indeed, let $A$ be the coordinate ring of an isotropic Grassmannian $X$ in its Pl\"ucker embedding, not equal to the symplectic Grassmannian $\LG(3,6)$ of $3$-planes in a $6$-dimensional vector space. Then $T^i_A = 0$
  for all $2\le i \le d-3$. We show that $T^{d-2}_A = 0$ if and only if $X$ is $\LG(n-1,2n)$ or $\OG(n,2n+1)$.
Furthermore, $T^1_A = 0$ as long as $X$ is not an isotropic Grassmannian of $1$ or $2$-planes, or $\OG(4,8)$.

This work was motivated by our attempt to understand the smoothings of
certain degenerate Fano varieties in homogeneous spaces. In our last
Section we give an application regarding deformations of
complete intersections in cones over Grassmannians.

\begin{ack} We  would like to thank Steven Sam for helpful discussions on
  representation theory. 
\end{ack}

\section{Preliminaries} 
\subsection{Cotangent cohomology}
To fix notation we give a short description of the cotangent modules
and sheaves. For definitions, proofs and details on this cohomology
and its relevance to deformation theory see \cite{an:hom},
\cite{il:com} and \cite{la:for}.  Given a ring $R$ and an $R$-algebra $S$ there is a complex of free $S$
modules; the {\it cotangent complex} $\mathbb L_{\bullet}^{S/R}$. See
e.g.\ \cite[p. 34]{an:hom} for a definition.  

For an $S$ module $M$ we
get the {\it cotangent cohomology} modules $$T^i(S/R;M)
=H^i(\Hom_S(\mathbb L_{\bullet}^{S/R},M))$$ and  {\it cotangent
  homology} modules  $$T_i(S/R;M)
=H_i(\mathbb L_{\bullet}^{S/R} \otimes_S M) \,.$$ If $R$ is the ground field
we abbreviate $T^i(S/R;M)=T^i_S(M)$ and
$T^i_S(S)=T^i_S$. Correspondingly we will write $T_i^{S/R}$ for the
homology. There is a natural spectral sequence
$$\Ext^p_S(T_q^{S/R},M)  \Rightarrow T^{p+q}(S/R;M) $$
which will in our case allow us to compute $T^i_A$ as
$\Ext_A^i(\Omega_A, A)$. See Proposition \ref{ext}.

\subsection{Representation theory} 
 We review  our notation and 
some theory which we have taken from
\cite{fh:rep}, \cite{we:coh} and \cite{rww:loc}. A weight of the maximal torus of
diagonal matrices in  $\GL_n$ is an $n$-tuple  $\lambda  =
(\lambda_1, \dots , \lambda_n) \in \ZZ^n$. It is \emph{dominant} if
$\lambda_1 \ge \lambda_2 \ge \dots \ge \lambda_n$. We will often use
the shorthand
$\lambda = (n_1^{a_1}, \dots , n_k^{a_k})$ meaning $n_i$ is repeated
$a_i$ times in the tuple. If $\lambda$  is a dominant weight with
$\lambda_n \ge 0$ then $\lambda$ yields a partition of $m = \sum
\lambda_i$ and we denote this $\lambda
  \vdash m$. If it is clear that $\lambda$ is a partition then we
do not include the trailing zeros in the tuple.

Given an $n$-dimensional
vector space $E$ the irreducible
representations of $\GL_n \simeq \GL(E)$ are in one-to-one 
correspondence with the dominant weights. We write $\Sch_{\lambda} E$
for the corresponding Schur functor, i.e.\ the irreducible
representation associated to $\lambda$. We have $\Sch_{(1^r)} E =
\wedge^r E$, $\Sch_ \lambda E \otimes \wedge^n E = \Sch_{\lambda +
  (1^n)}$ and $\Sch_ \lambda E^\ast = \Sch_{(-\lambda_n, \dots ,
  -\lambda_1)} E$.  
If $E$ and $F$ are vector spaces we have the Cauchy formula for  $\Sym^k(E
\otimes F)$ as $\GL(E) \times \GL(F)$-representation, namely
$$\Sym^k(E \otimes F) = \bigoplus_{\lambda \vdash k} \Sch_\lambda E
\otimes \Sch_\lambda F \, .$$
This and several other standard combinatorial statements (which may be
found in the above mentioned literature) relating to the
Littlewood-Richardson rule and Young diagrams are used in Section
\ref{abouthm}. 

\subsection{Bott's theorem for the Grassmannian}
Let $\GG = \GG(r,E)$ be the
Grassmannian of $r$-dimensional subspaces of $E$ and let  $$0 \to
\mathcal{R} \to 
\CO_\GG \otimes E \to \CQ \to 0$$ be the tautological 
sequence on $\GG$.  By functoriality the Schur functors may be applied
to vector bundles on the $\GG$, in particular to the tautological sub and
quotient bundles $\CR$ and $\CQ$. 

We review Bott's theorem applied to $\GG$ as
described in \cite[Section 2.2]{rww:loc}. It will be used in Section
\ref{abouthder}. Consider two dominant weights $\alpha
= (\alpha_1, \dots ,\alpha_{n-r})$ and $\beta = (\beta_1, \dots
,\beta_{r})$  and their concatenation $\gamma = 
(\gamma_1,\dots,\gamma_r)$. Let $\delta = (n-1, \dots , 0)$ and
consider $\gamma + \delta$. Write $\sort(\gamma + \delta)$ for the sequence obtained
by arranging the entries of $\gamma + \delta$ in non-increasing order, and define
$\tilde{\gamma} = \sort(\gamma + \delta)- \delta$.

\begin{theorem}[Bott] \label{bott} With the above notation,
if $\gamma + \delta$ has repeated entries, then
$$H^i(\GG,\Sch_{\alpha} \CQ \otimes \Sch_{\beta} \CR )=0$$ for all $i
\ge 0$. Otherwise,
writing $l$ for the number of pairs $(i, j)$ with $1 \le i < j  \le n$
and $\gamma_i - i
< \gamma_j - j$ , we have$$H^l(\GG,\Sch_{\alpha} \CQ \otimes
\Sch_{\beta} \CR ) = \Sch_{\tilde{\gamma}} E$$
and
$H^i(\GG,\Sch_{\alpha} \CQ \otimes \Sch_{\beta} \CR )=0$ for $i \ne
l$. 
\end{theorem}

We will also apply Bott's theorem to isotropic Grassmannians in Section \ref{sec:iso}. We refer the reader to \cite[4.3]{we:coh} for details.

\section{Computing higher cotangent cohomology}
We give here in successively more special cases the methods we will
use to compute the higher $T^i$.

\subsection{Cohen-Macaulay isolated singularities}

\begin{proposition} \label{ext} Let $(A,\mathfrak{m})$ be a
$d$-dimensional Cohen-Macaulay local $k$-algebra such that $\Spec A$
is an isolated singularity. Then
$$T^i_A\simeq \Ext_A^i(\Omega_A, A)$$ for
    $0 \le i \le d$.
\end{proposition}

\begin{proof} Consider the spectral sequence $\Ext^p_A(T_q^A, A)
\Rightarrow T^n_A$ and note that by the depth condition
$\Ext^p_A(T_q^A, A)$ vanishes if $q \ge 1$ and $p < d$.
\end{proof}

\begin{lemma} \label{HDer} Let $(A,\mathfrak{m})$ be a $d$-dimensional
Cohen-Macaulay local $k$-algebra such that $\Spec A$ is an isolated
singularity. Then
$$\Ext_A^i(\Omega_A, A) \simeq H^{i+1}_\mathfrak{m}(\Der_k(A))$$
  for $1 \le i \le d-2$.
\end{lemma}

\begin{proof} We will use $\Ext$ with support as described in SGA 2
Expos\`{e} VI (\cite{gr:sga2}), specifically
$\Ext^i_{\mathfrak{m}}(\Omega_A, A)$. Consider first the spectral
sequence $$\Ext^p_A(\Omega_A, H^q_{\mathfrak{m}} (A)) \Rightarrow
\Ext^n_{\mathfrak{m}}(\Omega_A, A)$$ which shows that
$\Ext^i_{\mathfrak{m}}(\Omega_A, A) = 0$ for $i < d$. There is a long
exact sequence
$$\cdots \to \Ext^i_{\mathfrak{m}}(\Omega_A, A)  \to
\Ext^i_A(\Omega_A, A) \to \Ext^i_{U}(\Omega_U, \CO_U) \to
\Ext^{i+1}_{\mathfrak{m}}(\Omega_A, A) \to \cdots$$ and it follows
that $\Ext^i_A(\Omega_A, A) \simeq \Ext^i_{U}(\Omega_U, \CO_U)$ for $i
\le d-2$. On the other hand $\Ext^i_{U}(\Omega_U, \CO_U) \simeq H^i(U,
\Theta_U)$ which again is isomorphic to
$H^{i+1}_\mathfrak{m}(\Der_k(A))$ for $i \ge 1$.
\end{proof}

\begin{proposition}\label{ticomp} Let $(A,\mathfrak{m})$ be a $d$-dimensional
Gorenstein local $k$-algebra with $d \ge 2$, such that $\Spec A$ is an
isolated singularity. Then
$$T^i_A\simeq \begin{cases} H^{i+1}_\mathfrak{m}(\Der_k(A)) \quad &\text{if
    $1 \le i \le d-2$ }\\ H^{1}_\mathfrak{m}(\Omega_A)^\ast \quad
&\text{if $i = d-1$ } \\ H^{0}_\mathfrak{m}(\Omega_A)^\ast \quad
&\text{if $i = d$.}
\end{cases}$$
\end{proposition}

\begin{proof} This follows directly from Proposition \ref{ext}, Lemma
\ref{HDer} and local duality.
\end{proof}

\subsection{Computing $H^i_\mathfrak{m}(\Der_k(A))$ for cones over
  projective schemes} 
Some of the ideas in this section were used by Svanes and Schlessinger
and may be found in
\cite{sv:som} and \cite{sc:rig}. We believe our approach is more direct and gives
more than the vanishing of the cohomology. To use Proposition
\ref{ticomp} we need to compute the local cohomology of the derivation
module. For cones over projective schemes $X$ we may relate this to
the sheaf cohomology of twists of $\CO_X$ and $\Theta_X$. 

Let $A$ be a standard graded $k$-algebra, i.e.\ the algebra generators
are in degree $1$. Let $X= \Proj A$ with irrelevant maximal ideal
$\mathfrak{m}$. Let 
$X^\prime = \Spec A \setminus V(\mathfrak{m})$, $\pi : X^\prime \to X$
the $\GG_m$ quotient and set $\CS = \pi_{\ast}\CO_{X^\prime}$ a sheaf
of graded algebras on $X$ with $\CS_0 = \CO_X$. Let $\Theta_ \CS$ the
sheaf which is locally $\Der_k(\CS(U))$ on $X$, i.e.\ $\Theta_ \CS =
\pi_\ast \Theta_{X^\prime}$. Then $\Theta_ \CS$ is a sheaf of graded
$\CS$-modules so let $\mathcal{E}$ be the degree 0 part.

If $\CS(U) = B$, so that $\CO_X(U) = B_0$ then the sequence
$$0 \to \Der_{B_0}(B) \to \Der_k(B) \to \Der_k(B_0,B) \to 0$$
is exact since $B$ is smooth over $B_0$. Moreover the Euler derivation
gives a graded isomorphism $B \simeq \Der_{B_0}(B)$. This globalizes
to an exact sequence $$0 \to \CS \to \Theta_\CS \to \Theta_X
\otimes_{\CO_X} \CS \to 0$$ and taking the degree $0$ part we
get \begin{equation}\label{eseq} 0 \to \CO_X \to \mathcal{E} \to
\Theta_X \to 0 \, .\end{equation} This sequence is locally
$$0 \to B_0\to \Der_k(B)_0 \to \Der_k(B_0) \to 0$$
so we see that $\mathcal{E} \simeq \widetilde{\Der_k(A)}$. Recall that
by comparing the \v{C}ech complex of 
$\widetilde{M}$ over $\Proj A$ and the complex computing $H^i_{\mathfrak{m}}(M)$
we get $\bigoplus_{m} H^i(\Proj A,
\widetilde{M}(m)) \simeq H^{i+1}_{\mathfrak{m}}(M)$ when $i\ge
1$. Thus we have proven 
\begin{lemma}\label{elemma} There are isomorphisms
$H^i_\mathfrak{m}(\Der_k(A)) \simeq \bigoplus_{m \in \ZZ}H^{i-1}(X,
\mathcal{E}(m))$ for $i \ge 2$.
\end{lemma}

\begin{proposition}[\cite{sv:som} Remark 2.5]\label{svanes}  Assume
  $X = \Proj A$ is smooth and 
  $1 \le j \le \dim X -1$. If $$H^i(X, \CO_X(m)) = H^i(X, \Theta_X(m)) = 0$$ for all $m$
  and all $1 \le i \le j$, then $T^i_A = 0$ for all $1 \le i \le j$.
\end{proposition}
\begin{proof} This follows  from Proposition \ref{ticomp}, the
  exact sequence \eqref{eseq} and Lemma \ref{elemma}.
\end{proof}

In our application we will need to prove that $T^{d-2}_A = 0$
even though not all $H^{d-2}(X, \Theta_X(m))$ vanish. For this we need to
understand the sequence \eqref{eseq} better. For any scheme there is a
natural map $\mathcal{O}_{X}^{*} \to 
\Omega^{1}_{X}$ defined locally by
$$
u \mapsto \frac{du}{u} \, .$$ Let $c: H^{1}(X, \mathcal{O}_{X}^{*})
\to H^{1}(X, \Omega^{1}_{X})$ be the induced map in cohomology. Now $
H^{1}(X, \Omega^{1}_{X}) \simeq \Ext^{1}(\mathcal{O}_{X},
\Omega^{1}_{X})$, so for a line bundle $L$, $c(L)$ gives an extension
\begin{equation*}
\label{ex} e_{L}: \quad 0 \to \Omega^{1}_{X} \to \mathcal{F}_{L} \to
\mathcal{O}_{X} \to 0 \, .
\end{equation*} Set $\mathcal{E}_{L} := \mathcal{F}_{L}^{\vee}$ and
note that the dual sequence
$$e_{L}^\vee: \quad
0 \to \mathcal{O}_{X} \to \mathcal{E}_{L} \to \Theta_{X} \to 0 $$ is
also exact. In the smooth case this is known as the Atiyah extension
associated to $L$, but we will call it that for general $X$.

We state
and prove for lack of reference (in this generality) the 
certainly well known
\begin{proposition} \label{at} If $X = \Proj A$ and $L = \CO_X(1) =
\widetilde{A(1)}$ then the sequence
$$0 \to \CO_X \to \mathcal{E} \to
\Theta_X \to 0$$ is the Atiyah extension $e_L^\vee$.
\end{proposition}
\begin{proof} Let $x_0, \dots x_n$ be a basis for $H^0(X, \CO_X(1))$
so we may realize $X$ in $\PP^n$. Set $B= A_{(x_0)} = k[x_0, \dots,
x_n, x_0^{-1}]/I$ for some ideal $I$. Then $B_0= k[y_1, \dots ,
y_n]/J$ where $J$ is generated by the $f(1, y_1, \dots, y_n)$ with $f
\in I$ and the inclusion is given by $y_i \mapsto x_ix_0^{-1}$. For a
homogeneous $f \in B_d$
\begin{equation}\label{dehom} f(x_0, \dots , x_n) = x_0^df(1,
\frac{x_1}{x_0}, \dots, \frac{x_n}{x_0})
\end{equation}

Write $\partial_x$ for the partial derivative of a variable $x$. A
derivation $D \in \Der(B_0)$ can be written $D = \sum_i
a_i \partial_{y_i}$ where the $a_i$ are such that $D(f)= 0$ in $B_0$
for all $f \in J$. From $D$ we can form
$$\widetilde{D} = \sum_{i=1}^n
  a_i(\frac{x_1}{x_0}, \dots, \frac{x_n}{x_0}) \cdot
x_0 \partial_{x_i}\, .$$ Using \eqref{dehom} one can check that
$\widetilde{D}$ is a well defined derivation of $B$. It is clearly of
degree $0$. Moreover one may compute that for $D \in \Der(B)_0$
$$D - \widetilde{D_{\mid B_0}} = g \sum_{i=0}^n x_i \partial_{x_i}$$
for suitable $g$.  This implies that \eqref{eseq} is locally split.

The sequence $e^\vee_{L}$ is also locally split and we may write
$\mathcal{E}_{L}$ locally on $U_{i}$ as $\mathcal{O}_{U_{i}} \oplus
\Theta_{U_{i}}$. Let $L$ be represented by a \v{C}ech cocycle
$(f_{ij})$, $f_{ij} \in \Gamma(U_{ij}, \mathcal{O}^{*}_{X})$.  The
gluing of $\mathcal{E}_{L}$ is determined (dually) by the extension
class in $H^{1}(\Omega^{1}_{X})$; $(g_{i}, D_{i}) \in \Gamma(U_{i},
\mathcal{E}_{L})$ and $ (g_{j}, D_{j}) \in \Gamma(U_{j},
\mathcal{E}_{L})$ are equal on $U_{ij}$ iff $D_{i} = D_{j}$ and $g_{j}
- g_{i} = D_{i}(f_{ij})/f_{ij}$. Now use the above local splitting to
show that when $L=\CO_X(1)$ we have $\mathcal{E} \simeq
\mathcal{E}_{L}$.
\end{proof}

\subsection{Computing $H^0_\mathfrak{m}(\Omega_{S^G})$ and
$H^1_\mathfrak{m}(\Omega_{S^G})$ for invariant rings} \label{genHm}
Let $S$ be a finitely generated standard graded $k$-algebra with the action of a
linearly reductive group $G$ respecting the grading. Assume $\Spec
S^G$ has an isolated 
singularity at $\mathfrak{m} \subset S^G$. If $J = \mathfrak{m}S$
assume that $\depth_J 
S\ge 2$ and that  $\depth_J 
\Omega_S\ge 2$. 

Let $H$ be the kernel of the map $\Omega_{S^G} 
\otimes_{S^G} S \to \Omega_S$.  
\begin{lemma} There are isomorphisms $$H^0_\mathfrak{m}(\Omega_{S^G})
\simeq H^G \quad \text{and} \quad
H^1_\mathfrak{m}(\Omega_{S^G}) \simeq H^0_J(\Omega_{S/S^G} )^G \, .$$
\end{lemma}
\begin{proof} Consider the exact sequence
\begin{equation} \label{es1} 0 \to H \to \Omega_{S^G} \otimes_{S^G} S \to \Omega_S \to
\Omega_{S/S^G} \to 0 \end{equation} and note that $(\Omega_{S^G} \otimes_{S^G} S)^G =
\Omega_{S^G}$. 
 We split the sequence into 2 short exact
sequences.
On the right we get
$$0 \to K \to \Omega_S \to
\Omega_{S/S^G} \to 0$$ which yields $H^0_J(K) = 0$ and $H^1_J(K)
\simeq H^0_J(\Omega_{S/S^G} )$. On the left we get
$$0 \to H\to \Omega_{S^G} \otimes_{S^G} S  \to
K \to 0\, .$$ The module $H$ is supported at $J$ so
$H^i_J(H) = 0$ for $i \ge 1$ and the sequence yields
$H \simeq H^0_J(H) \simeq H^0_J(\Omega_{S^G} \otimes_{S^G} S)$ and
$H^0_J(\Omega_{S/S^G} )  \simeq H^1_J(K) \simeq H^1_J(\Omega_{S^G}
\otimes_{S^G} S)$. Taking invariants yields the result.
\end{proof}

A series  of right exact sequences 
$$B^{i-1} \xrightarrow{\beta_{i-1}} C^i \xrightarrow{\gamma_{i}}  B^i
\to 0$$
leads to a complex 
$$\cdots \xrightarrow{\delta_{i-2}}  C^{i-1} \xrightarrow{\delta_{i-1}}   C^{i}
\xrightarrow{\delta_{i}}  C^{i+1} \xrightarrow{\delta_{i+1}}  \cdots$$
with $\delta_i = \beta_{i}\circ \gamma_i$. Moreover since the
sequences are right exact we have $H^i(C^\bullet) \simeq \Ker
\beta_i$. We will use this construction to get a four term complex
which computes the local cohomology we are interested in.

Let $\mathfrak{g}$ be the Lie algebra of $G$. By \cite[Lemma
4.7]{ck:com} there are isomorphisms  
$$\Hom_S(\Omega_{S/S^G}, S) \simeq \Der_{S^G}(S) \simeq S 
\otimes \mathfrak{g}\, .$$ Choosing a basis for $\mathfrak{g}$ defines a
$G$-equivariant 
map $\Omega_{S/S^G} \xrightarrow{E} S \otimes \mathfrak{g}^\ast \simeq
j_\ast j^\ast\Omega_{S/S^G}$ where $j$ is the inclusion of $\Spec S
\setminus V(J)$ in $\Spec S$ (see \cite[Section 4.2]{ck:com}).  Thus
we have an exact sequence
\begin{equation}\label{es2}0 \to H^0_J(\Omega_{S/S^G}) \to
\Omega_{S/S^G} \xrightarrow{E} S \otimes \mathfrak{g}^\ast \to
H^1_J(\Omega_{S/S^G}) \to 0\, .\end{equation} Assume that the algebra
generators of $S^G$ are in a single degree in $S$, i.e.\ that they generate
a subspace $U^\ast$ of a certain $S_r$.  
The invariant polynomials
define an embedding $\Spec S^G \subset U$. 

Set $P= \Sym U^\ast$ and
let $I$ be the kernel of $P \to S^G$. Assume that the generators of
$I$ are in a single degree and span a subspace $F \subseteq P_s$. Now
$\Omega_P \otimes_P S^G \otimes_{S^G} S \simeq P \otimes_k U^\ast
\otimes_P S \simeq S \otimes_k U^\ast$ and $ I/I^2 \otimes_{S^G} S
\simeq I \otimes_P S$ is the image of $S \otimes F$ so we get an exact
sequence 
\begin{equation}\label{es3} S \otimes F \to  S
\otimes U^\ast \to \Omega_{S^G}\otimes_{S^G} S
\to 0 \, .\end{equation}
We construct our complex from the right parts of the sequences
\eqref{es1} and \eqref{es2} together with \eqref{es3}. We put
everything into a diagram with exact rows and columns. The complex
then consists of the diagonal maps in
\begin{equation} \label{diag}
\begin{tikzcd} 0 \arrow{r}  & S\otimes F \arrow{r} \arrow{rd} & S \otimes F
  \arrow{d} \arrow{r} & 0 \\
 &&S \otimes U^\ast \arrow{d} \arrow{rd}\\  
&&\Omega_{S^G}\otimes_{S^G} S \arrow{d} \arrow{r} &\Omega_S \arrow{rd}\arrow{r}&
\Omega_{S/S^G} \arrow{r} \arrow{d}& 0
 \\ 
&&0 &&  S \otimes \mathfrak{g}^\ast \arrow{d}\\
&&&& H^1_J(\Omega_{S/S^G}) \arrow{d} \\
&&&& 0
 \end{tikzcd}
\end{equation}
so we have proven

\begin{proposition} \label{cplx} The four term complex
$$C^\bullet: \quad (S\otimes F)^G \xrightarrow{d^1}  (S\otimes
U^\ast)^G \xrightarrow{d^2}  
(\Omega_S)^G  \xrightarrow{d^3}  (S \otimes \mathfrak{g}^\ast)^G$$  has 
$H^1(C^\bullet) \simeq H^0_\mathfrak{m}(\Omega_{S^G})$ and
$H^2(C^\bullet) \simeq H^1_\mathfrak{m}(\Omega_{S^G})$.
\end{proposition}

To apply this we will need a more detailed description of $d^3$  in
the case when  $S = \Sym V^\ast$ for a $G$-representation $V$. Let
$x_1, \dots , x_n$ be a basis for $\Sym^1 V^\ast$. 
We start with the dual cotangent sequence for $k \to S^G \to
S$, i.e.\
$$0 \to \Der_{S^G}(S) \to \Der_k(S) \to \Der_k(S^G, S)$$
which under the assumptions is also right exact (see e.g.\
\cite[Section 4.2]{ck:com}). We have (see above) $\Der_{S^G}(S) \simeq
S \otimes \mathfrak{g}$ and we always have $\Der_k(S)
\simeq S \otimes_k V$ using $\frac{\partial}{\partial x_i}$ as a basis
for $V$. 

Let $$\rho: \mathfrak{g} \to \Sym^1 V^\ast
\otimes V \simeq \Hom(V,V)$$ be the induced
representation of the Lie algebra.
On graded pieces we have 
$$S \otimes \mathfrak{g} \simeq \Der_{S^G}(S) \to \Der_k(S) \simeq S
\otimes V$$
given by the composite
$$\Sym^k V^\ast \otimes \mathfrak{g} \xrightarrow{\id \otimes \rho} \Sym^{k} V^\ast 
\otimes \Sym^{1} V^\ast \otimes V \xrightarrow{\mu \otimes \id } \Sym^{k+1} V^\ast \otimes V$$
where $\mu$ is multiplication.

It will be convenient to express $\rho$ using the basis $\{x_i\}$ for
$V^\ast$ so write $$\rho(X) = \sum \rho_i(X) \otimes
\frac{\partial}{\partial x_{i}}\, .$$
We
have a composite map
$$\mathfrak{g} \otimes V \xrightarrow{\rho \otimes \id} \Hom(V,V)
\otimes V  \xrightarrow{c} V$$
where $c$ is the contraction $c(\varphi \otimes v) = \varphi(v)$. Let
$\beta: V^\ast \to \Hom(\mathfrak{g}, V^\ast)$ be the dual, i.e\
$\beta(\psi)(X)=\psi \circ \rho(X)$. 

The map $d^3: \Omega_S \simeq S \otimes V^\ast
\to S \otimes \mathfrak{g}^\ast$ on graded pieces is the
composition 
$$S_k \otimes V^\ast \xrightarrow{\id \otimes \beta} S_k \otimes
\Sym^1 V^\ast
\otimes \mathfrak{g}^\ast \xrightarrow{\mu \otimes \id}  S_{k+1}
\otimes \mathfrak{g}^\ast$$
where $\mu$ is the multiplication map.
If we use the identification $S
\otimes \mathfrak{g}^\ast \simeq \Hom(\mathfrak{g}, S)$ we get 
\begin{equation}\label{betaeq} d^3(f
\otimes \psi)(X) = f \sum_j \psi(\frac{\partial}{\partial x_{j}})
\rho_j(X) \, .
\end{equation}
In particular $d^3(dx_i)(X) = \rho_{i}(X)$.

\section{Cotangent cohomology of  Pl\"ucker algebras} \label{grass}
Let $E$ be an $n$-dimensional vector space and $\GG = \GG(r,E)$ the
Grassmannian of $r$-dimensional subspaces. Let $A$ be the homogeneous
coordinate ring of $\GG$ in the Pl\"ucker embedding.
Fix an $r$-dimensional vector space $W$ and consider $V = \Hom(W,E)$
which we may think of as the space of $n \times r$ matrices. We have
the natural action of $\GL(E) \times \GL(W)$ on $V$ and $V^\ast =
E^\ast \otimes W$. 

For this section  set $G = \SL(W)$ and $S = \Sym V^\ast$ so that $A =
S^G$. Set $d = \dim A = (n-r)r + 1$. We write $$S =k[x_{ij}: 1 \le i \le n, 1 \le j \le r]$$ where after
fixing basis $\{e_i\}$ for $E$ and $\{w_j\}$ for $W$ we have $x_{ij} =
e_i^\ast \otimes w_j$.

Set $U^\ast = \bigwedge^r E^\ast \otimes \bigwedge^r W
\subset \Sym^r(E^\ast \otimes W)$. Then a basis for $U^\ast$ form the
generators of  $J$, the ideal of maximal $r
\times r$ minors in a general $n \times r$ matrix and they generate
the algebra $A = S^G$. If $P = \Sym U^\ast$ then the kernel $I$ of the
surjection $P \to A$ is generated by the quadratic Pl\"ucker
relations.  

Combining Proposition \ref{ticomp} with
Proposition \ref{him}, Theorem \ref{h0mo} and Proposition
\ref{h1mo}, which are proven below, we get the following theorem
\begin{theorem} \label{main} Assume $A$ is the Pl\"ucker algebra for a
  Grassmannian $\GG(r,n)$ different from $\GG(2,4)$. Then  $T^i_A = 0$
  for $1 \le i\leq  d-1 =n(n-r)$ and 
  $T^{d}_A = 0$ if and only if $r=2$ or
  $r=n-2$.
\end{theorem}

\begin{remark}  If  $r \ne 2$ and 
  $r \ne n-2$ then $T^{d}_A$ is concentrated in
  degree $2$, see Theorem \ref{h0mo} below.
\end{remark}

The result is sharp, i.e.\ we cannot expect 
$T_A^{d+1} \ne 0$ as seen in this example.
\begin{example} Let $A$ be the Pl\"ucker algebra for $\GG(2,6)$ of
  dimension $9$. Let
$p_{ij}$, $1 \le i < j \le 6$ be the Pl\"ucker coordinates.  The ideal generated by
$$p_{12}, p_{23},  p_{34},  p_{45},  p_{56},  p_{16},
p_{14}+p_{34}+p_{26},  p_{24}+p_{15}+p_{36}$$
defines a codimension 8 complete intersection ideal in $A$. Let $B$ be
the coordinate ring of this curve.  A Macaulay2 computation shows that $\dim T_B^2
= 1$. By \cite[1.4.2]{bc:hyp} this implies that $T^{10}_A\neq 0$.
\end{example}

\subsection{About $H^i_\mathfrak{m}(\Der_k(A))$.} \label{abouthder}
Let $$0 \to \mathcal{R} \to
\CO_\GG \otimes E \to \CQ \to 0$$ be the tautological
sequence on $\GG$. Recall that $\Theta_\GG \simeq \mathcal{R}^\vee
\otimes_{\CO_\GG} \CQ$ and that $\CO_\GG(m) \simeq
(\wedge^{n-r}\CQ)^{\otimes m}$. 
\begin{lemma} \label{botttheta} There are isomorphisms of
$\SL(E)$-modules
$$H^i(\GG, \CO_\GG(m)) \simeq \begin{cases}\Sch_{(
    m^{n-r})} E &\text{if $i=0$ and $m \ge 0$} \\
\Sch_{((-m-n)^{r})} E&\text{if $i = r(n-r)$ and $m \le -n$} \\ 0
&\text{for all other values of $i$ and $m$}
\end{cases}$$ and if $(r,n) \ne (2,4)$ then 
$$H^i(\GG, \Theta_\GG(m)) \simeq \begin{cases}\Sch_{(m+1,
    m^{n-r-1},0^{r-1},-1)} E &\text{if $i=0$ and $m \ge 0$} \\
\Sch_{(0)} E &\text{if $i = r(n-r)-1$ and $m = -n$} \\
\Sch_{((-m-n)^{r-1}, -m-n-1, 1)} E &\text{if $i = r(n-r)$ and $m
\le -n-2$} \\ 0 &\text{for all other values of $i$ and $m$.}
\end{cases}$$
\end{lemma}
\begin{proof} We use Bott's theorem as described in Theorem
  \ref{bott}. We only give the calculation for $$\Theta_\GG(m) 
  \simeq \Sch_{(m+1, m^{n-r-1})}\CQ \otimes \Sch_{( 0^{r-1},
    -1)}\CR \, .$$ Let $\lambda = (m+1, m^{n-r-1},
0^{r-1}, -1)$. If $\delta = (n-1, \dots, 0)$ then $$\lambda + \delta =
(m+n, m+n-2, \dots , m+r, r-1, \dots , 1,-1)$$
cannot have repeated entries if  $m > -1$, $m < -n-1$ or $m= -n$. On
the other hand one can easily check that if if $m= -n-1$ or $-n+1  
\le m \le -n+r-1$ then $m+n$ is repeated. If
$-n+r+1 \le m \le -1$ then $m = -n + r +k$ with $1 \le k \le n-r-1$
so $m+n-2 \ge m + n-(k+1) \ge m+r$. Thus $m+n-(k+1) = r-1$ is
repeated. If finally $m= -n+r$ assume first that $r \ge 3$. Then
$m+n-2 = r-2 \ge 1$ so it is repeated. If $r=2$ and $n \ge 5$ then
$n-3 \ge r$ so $m + (n-3) = -1$ is repeated. We conclude that if
$(r,n) \ne (2,4)$ then $H^i(\GG, \Theta_\GG(m)) = 0$ for all values of
$i$ if and only if $-n+1 \le m \le -1$ or $m=-n-1$. 

If $m \ge 0$ then $\lambda + \delta$ is
non-decreasing so $H^0(\GG,
\Theta_\GG(m))  \simeq \Sch_\lambda E$ and all other cohomology vanishes. 
If $m \le -n-2$ then $\lambda + \delta$ needs $r(n-r)$
adjacent transpositions
to become the non-decreasing
  $$(r-1, r-2, \dots , 1, -1, m+n, m+n-2, \dots , m+r)\, .$$
Subtracting $\delta$ we get $$((r-n)^{r-1}, r-1-n,m+r+1,
(m+r)^{n-r-1})$$ so the only non-zero
cohomology is $H^{r(n-r)}(\GG, \Theta_\GG(m)) \simeq
\Sch_{((-m-n)^{r-1}, -m-n-1,1)} E$ as $\SL(E)$-modules. 
If $m=-n$ then $\lambda + \delta$ needs $r(n-r)-1$ adjacent
transpositions to become 
$$( r-1, \dots , 1 ,0,-1, -2, \dots , -n+r)\, .$$
Subtracting $\delta$ we get $((r-n)^{n})$ so the only non-zero
cohomology is $H^{r(n-r)-1}(\GG, \Theta_\GG(-n)) \simeq
\Sch_{(0)} E \simeq k$ as $\SL(E)$-modules. 
 \end{proof}

\begin{remark} If we do the above calculation for $\Theta_{\GG}(m)$ on
  $\GG(2,4)$ we get $$\lambda + \delta = (m+4,m+2,1,-1)$$
which has repeated entries iff $m$ equals $-1$, $-3$ or $-5$. Thus in
addition to the cohomology described in the lemma, we must check when
$m=-2$. Then $\lambda + \delta$ needs one adjacent transposition to
become $(2, 1, 0, -1)$ and subtracting $\delta$ we get $(-1, -1, -1,
-1)$. Thus the isomorphism $H^1(\Theta(-2)) \simeq k$ corresponds to $(T^1_A)_{-2}
\simeq k$. 
\end{remark}

\begin{proposition} \label{him} If $A$ is the Pl\"ucker algebra
for $\GG(r,n)$ which is not $\GG(2,4)$, then $$H^i_\mathfrak{m}(\Der_k(A)) = 0$$ for $0
\le i \le d-1$.
\end{proposition}
\begin{proof}Since
$\depth_\mathfrak{m}A \ge 2$, the module
$H^i_\mathfrak{m}(\Der_k(A)) = 0$ for $i=0,1$. The vanishing of 
$H^i_\mathfrak{m}(\Der_k(A))$ for $i=2, \dots , d-2$ follows from the
sequence \eqref{eseq}, Lemma \ref{elemma} and Lemma \ref{botttheta}.

To show that $H^{d-1}_\mathfrak{m}(\Der_k(A)) = 0$ we must show that
the connecting map $$H^{d-2}(\GG, \Theta_\GG(-n)) \to H^{d-1}(\GG,
\CO_\GG(-n))$$ is injective. Note that $\omega_\GG = \CO_\GG(-n)$.
Now we know from Proposition \ref{at} that \eqref{eseq} is the Atiyah
extension, so by Serre duality this will follow if the connecting map
$H^0(\GG, \CO_\GG) \to H^1(\GG, \Omega_\GG)$ from $e_L$ is an
isomorphism.  This is the map \begin{equation*} \Hom_\GG(\CO_\GG,
\CO_\GG) \xrightarrow{\gamma} \Ext_\GG^1(\CO_\GG,
\Omega_\GG)\end{equation*} from the long exact $\Ext$-sequence of
$e_L$.  Recall that if
$$e: 0 \to A \to B \to C \to 0$$ is an exact sequence then the induced
map $\Hom(M,C) \to \Ext^1(M,A)$ sends $\varphi$ to the class of the
pullback over $\varphi$ of $e$. Thus $\gamma(\id)$ is the class of
$e_L$ and $\gamma$ is an isomorphism.
\end{proof}

\subsection{About $H^0_\mathfrak{m}(\Omega_{A})$ and
$H^1_\mathfrak{m}(\Omega_{A})$.} \label{abouthm}
We now compute $H^0_\mathfrak{m}(\Omega_A)$ and
$H^1_\mathfrak{m}(\Omega_A)$ using the complex 
$$C^\bullet: \quad (S\otimes F)^G \xrightarrow{d^1}  (S\otimes
U^\ast)^G \xrightarrow{d^2}  
(\Omega_S)^G  \xrightarrow{d^3}  (S \otimes \mathfrak{g}^\ast)^G$$ 
of Proposition \ref{cplx}  and the $\GL(E) \times \GL(W)$ action on
everything. Let us first identify the representations corresponding to
the modules involved. 

If we use
the $P$-grading on $S^G$ we have $$S^G = \bigoplus_{m \ge 0}S^G_m
\simeq \bigoplus_{m \ge 0}\Sch_{(m^r)} E^\ast\, .$$ 
The ideal $I$ generated by the Pl\"ucker relations in
$P_2$ is generated by $$F \simeq
\bigoplus_{\substack{2 \le i \le \min(r,n-r) \\ i \text{ even}}}
\Sch_{(2^{r-i}, 1^{2i})} E^\ast$$ (see e.g.\ \cite[Exercise
15.43]{fh:rep}). Thus the graded pieces of 
$(S\otimes F)^G = S^G \otimes F$ are \begin{equation}\label{ff}
  \bigoplus_{\substack{2 \le i 
    \le \min(r,n-r) \\ i \text{ even}}}\Sch_{(m^r)} 
E^\ast \otimes
\Sch_{(2^{r-i}, 1^{2i})} E^\ast\end{equation}
for each $m \ge 0$.
We have the graded pieces of $(S\otimes U^\ast)^G = S^G \otimes
U^\ast$ given as
\begin{equation} \label{uu} S^G_m \otimes U^\ast \simeq
\Sch_{(m^r)} E^\ast \otimes \Sch_{(1^r)} E^\ast \simeq \bigoplus_{0 \le i \le
\min(r,n-r)} \Sch_{((m+1)^{r-i}, m^{i}, 1^{i})} E^\ast \end{equation} (see
e.g.\ \cite[\textsection 6.1 (6.9)]{fh:rep}). 

Now $\Omega_S \simeq S \otimes V^\ast = S \otimes W \otimes E^\ast$ as
$\GL(E) \times \GL(W)$ module. The $S$-graded pieces are $$\Sym^k (W
\otimes E^\ast) \otimes W \otimes E^\ast \simeq \bigoplus_{\lambda
  \vdash k} (\Sch_{\lambda} E^\ast
\otimes E^\ast) \otimes (\Sch_{\lambda} W \otimes W) \, .$$
The
only $\lambda$ for which $\Sch_{\lambda} W \otimes W$ contains an
$\SL(W)$ invariant subspace are $$\lambda = ((m+1)^{r-1}, m)$$ in
degree $k =
(m+1)r-1$ for some $m \ge 0$. The invariant part is
\begin{equation} \label{vv} \begin{split}(S_{(m+1)r-1} \otimes
V^\ast)^G &\simeq \Sch_{ ((m+1)^{r-1}, m)} E^\ast \otimes E^\ast \\
&\simeq \begin{cases}
\Sch_{ (m+2,(m+1)^{r-2}, m)} E^\ast \oplus \Sch_{ ((m+1)^{r})} E^\ast \oplus
\Sch_{ ((m+1)^{r-1}, m,1)} E^\ast  \text{ if $m \ge 1$}\\
\Sch_{ (2,1^{r-2})} E^\ast \oplus \Sch_{ (1^{r})} E^\ast
\text{ if $m = 0$}
\end{cases} \end{split}
\end{equation} as $\GL(E)$ representation.

We identify  $\mathfrak{g}^\ast = \mathfrak{sl}_r^\ast \simeq
\Sch_{(2, 1^{r-2})} W \otimes \bigwedge^r W^\ast$ so 
$$S_{k} \otimes \mathfrak{g}^\ast \simeq \bigoplus_{\lambda
\vdash k} \Sch_{\lambda} E^\ast \otimes (\Sch_{\lambda} W \otimes
\Sch_{(2, 1^{r-2})} W \otimes \bigwedge^r W^\ast)\, .$$ The only 
$\lambda$ where $\Sch_{\lambda} W \otimes \Sch_{(2, 1^{r-2})} W$
contains an $\SL(W)$ trivial representation are $$\lambda = (m+2,
(m+1)^{r-2}, m)$$ in degree $k = (m+1)r$ for $m\ge 0$. So the
invariant part is 
\begin{equation}\label{gg} (S_{(m+1)r} \otimes \mathfrak{sl}_r^\ast)^G
\simeq \Sch_{(m+2, (m+1)^{r-2}, m)} E^\ast\end{equation} as $\GL(E)$
representation.

The map $d^2: \Omega_P \otimes_P S^G \to \Omega_S^G$ is induced by the
Jacobian matrix of the generators of 
$S^G$, i.e.\ the $r \times r$ minors. It has therefore $S$-degree $r-1$.  Let
$d^2_m$ be the map on graded pieces $(S_{mr})^G \otimes U^\ast \to 
(S_{(m+1)r-1} \otimes V^\ast)^G$.

\begin{lemma} \label{kerd} If $m \ge 1$, $$\Image d^2_m
 \simeq \Sch_{((m+1)^{r})} E^\ast  \oplus \Sch_{((m+1)^{r-1}, m, 1)} E^\ast $$  and $$\Ker d^2_m
\simeq \bigoplus_{2 \le i \le \min(r,n-r)} \Sch_{((m+1)^{r-i}, m^{i},
1^{i})} E^\ast$$
as $\GL(E)$ representations.
\end{lemma}
\begin{proof} Comparing \eqref{uu} and \eqref{vv} we see that the
  second statement follows from the first and that we must
  show that the endomorphisms on $\Sch_{ 
((m+1)^{r})} E^\ast$ and $\Sch_{ ((m+1)^{r-1}, m,1)} E^\ast$ induced by $d^2$ are
isomorphisms. By Schur's Lemma it is enough that $d^2$ is non-zero on
them. Let $u_1, u_2 \in U^\ast$ be $u_1 = |x_{ij}|$ for $1 \le i,j \le
r$ and $u_2 = |x_{ij}|$ for $i = 1, \dots , r-1, r+1$ and $1 \le j \le
r$. Thus $u_1 \mapsto e^\ast_1 \wedge \cdots \wedge e^\ast_r$ and $u_2 \mapsto
e^\ast_1 \wedge \cdots \wedge e^\ast_{r-1} \wedge e^\ast_{r+1}$ via $U^\ast \simeq
\wedge^r E^\ast$.

The part $\Sch_{ ((m+1)^{r})} E^\ast \simeq S_{m+1}^G \subset S_m^G \otimes
U^\ast$ corresponds to $\{df : f \in S_{m+1}^G \}$ and clearly $d^2$ is
non-zero on this. Indeed,  the image of the highest weight vector $u_1^m \otimes
du_1$ is clearly non-zero. It is easily seen that $u_1^m \otimes du_2$ is a weight vector
for the highest weight $((m+1)^{r-1}, m,1)$, so $u_1^m \otimes du_2$
is in the $\Sch_{ ((m+1)^{r-1}, m,1)} E^\ast$ part and does not map to $0$.
\end{proof}

\begin{theorem} \label{h0mo} If $A$ is the Pl\"ucker algebra for $\GG(r,E)$ with
  $\dim E = n$, then
  $H^0_\mathfrak{m}(\Omega_A)$ vanishes if and only  if $r=2$ or
  $r=n-2$. If  $r \ne 2$ and 
  $r \ne n-2$ then $H^0_\mathfrak{m}(\Omega_A)$ is concentrated in
  degree $2$ and 
$$H^0_\mathfrak{m}(\Omega_A)_2\simeq \bigoplus_{\substack{3 \le i \le \min(r,n-r) \\ i \text{ odd }}}
\Sch_{(2^{r-i}, 1^{2i})} E^\ast$$
 and is therefore the kernel of the
projection  $\wedge^2(\wedge^r E^\ast) \twoheadrightarrow \Sch_{(2^{r-1},
  1^{2})} E^\ast$.
\end{theorem}
\begin{proof} Since the  Pl\"ucker relations are in degree 2, $d^1$ in
  the $P$-grading take $S^G_m
\otimes F$ to $S^G_{m+1} \otimes U^\ast$. If $m=0$ we get a map to
$\Ker d_1^2$, i.e.\ from \eqref{ff} and Lemma \ref{kerd} a
map \begin{equation} \label{step1} 
  \bigoplus_{\substack{2 \le i \le \min(r,n-r) \\ i 
\text{ even}}} \Sch_{(2^{r-i}, 1^{2i})} E^\ast \to \bigoplus_{2 \le i \le
\min(r,n-r)} \Sch_{(2^{r-i}, 1^{2i})} E^\ast \end{equation} which cannot be surjective
unless $r$ or $n-r$ equals 2. The map is $f \mapsto df$ which cannot
be $0$ on the generators of $I$, so by Schur's Lemma \eqref{step1} is
injective.
Thus $H^0_{\mathfrak{m}}(\Omega_A)_2 = 0$  only for $r=2$ or
$n-r=2$. 
Moreover 
$$\wedge^2(\wedge^r E^\ast) \simeq \bigoplus_{\substack{1 \le i \le \min(r,n-r) \\ i \text{ odd }}}
\Sch_{(2^{r-i}, 1^{2i})} E^\ast$$
 (see e.g.\ \cite[Exercise
15.32]{fh:rep}), so if $r$ and $n-r$ do not equal 2 then 
$H^0_{\mathfrak{m}}(\Omega_A)_2$ is isomorphic to the kernel of the
projection  $\wedge^2(\wedge^r E^\ast) \twoheadrightarrow \Sch_{(2^{r-1},
  1^{2})} E^\ast$. 

On the other hand we claim that when $m \ge 1$ the map $S^G_m \otimes F
\to \Ker d_{m+1}^2$ is surjective. We first check that the
$\Sch_{((m+2)^{r-i}, (m+1)^{i}, 1^{i})} E^\ast$ for $2 \le i \le
\min(r,n-r)$ all appear as summands in
$$S^G_m \otimes F \simeq \bigoplus_{\substack{2 \le i \le \min(r,n-r) \\ i \text{ even}}}
\Sch_{(m^{r})} E^\ast \otimes \Sch_{(2^{r-i}, 1^{2i})} E^\ast\, .$$ Indeed,  if
$i$ is even and $2 \le i \le \min(r,n-r)$ then an application of the
Littlewood-Richardson rule shows that both  
$$\Sch_{((m+2)^{r-i}, (m+1)^{i},1^{i})} E^\ast \quad \text{and} \quad \Sch_{((m+2)^{r-(i+1)},
(m+1)^{i+1},1^{i+1})} E^\ast$$ appear in the decomposition of
$\Sch_{(m^{r})} E^\ast \otimes \Sch_{(2^{r-i}, 1^{2i})} E^\ast$. 

We must now show that the induced endomorphisms of the
$\Sch_{((m+1)^{r-i}, m^{i},1^{i})} E^\ast$ are isomorphisms. We do this
by induction on $m$. The map  $S^G_m \otimes F
\to S^G_{m+1} \otimes U^\ast$ factors through $(
I/I^2)_{m+2}$. Let $u_0$ be a Pl\"ucker coordinate and assume $f \in 
(I/I^2)_m$ with $df \ne 0$ in $S^G \otimes U^\ast$. Then $d(u_0 f) = u_0df
\ne 0$ in $S^G \otimes U^\ast$. 

If $m=1$ let
$f$ be a Pl\"ucker relation in $\Sch_{(2^{r-i}, 1^{2i})} E^\ast$ with $i$
even. Let $u_0$ correspond to $e^\ast_1 \wedge \dots \wedge e^\ast_r$ and  $u_1$
correspond to $e^\ast_1 \wedge \dots \wedge \widehat{e^\ast_{r-i}} \wedge \dots
\wedge e^\ast_{r+1}$. Then $u_0 \otimes f \in \Sch_{(3^{r-i}, 2^{i},1^{i})}
E^\ast \subset  S^G_1 \otimes F$ and $u_1 \otimes f \in \Sch_{(3^{r-(i+1)}, 2^{i+1},1^{i+1})}
E^\ast$ and by the above they do not map to $0$. Now assume the maps are
isomorphisms up to degree $m$. Let $f \in 
(I/I^2)_{m+2}$ be the image of something in $\Sch_{((m+2)^{r-i},
  (m+1)^{i}, 1^{i})} E^\ast$. Then $u_0f$ is the image of something in $\Sch_{((m+3)^{r-i},
  (m+2)^{i}, 1^{i})} E^\ast$ and by the above does not map to $0$.
\end{proof}

\begin{remark} The statement about $H^0_{\mathfrak{m}}(\Omega_A)_2$
  follows for more general reasons from the fact that it is the
kernel of the Gaussian map $\wedge^2 H^0(X,L) \to H^0(X, \Omega^1_X
\otimes L^2)$ for $L = \CO_X(1)$ (\cite[Propositions 1.4 and
1.8]{wa:coh}). Our result on the vanishing of
$H^0_{\mathfrak{m}}(\Omega_A)_m$ for $m \ne 2$ yields an affirmative
answer to the question 
  \cite[Problem 2.7]{wa:coh} by  Jonathan Wahl in the case $G/P$ is
  a Grassmannian. 
\end{remark}

The map $d^3: \Omega_S^G \simeq S \otimes V^\ast \to  (S \otimes
\mathfrak{sl}_r^\ast)^G$ on graded pieces is 
$$d^3_m : (S_{mr-1} \otimes
V^\ast)^G \to (S_{mr} \otimes \mathfrak{sl}_r^\ast)^G$$ for $m\ge 1$.
To continue we will need $\SL(W)$-invariants in $\Omega_S$.  To make such, take
an $r \times r$ submatrix of $(x_{ij})$ and replace one of the rows
with the tuple $(dx_{p, 1}, dx_{p,2}, \dots , dx_{p,r})$. Now take the
determinant to get an $\SL(W)$-invariant differential form.
The special invariant form 
$$\delta
=\begin{vmatrix}
x_{1,1} &  x_{1,2}  & \ldots & x_{1,r}\\
x_{2,1}   &  x_{2,2} & \ldots & x_{2,r}\\
\vdots & \vdots & \ddots & \vdots\\
x_{r-1,1}   &  x_{r-1,2} & \ldots & x_{r-1,r}\\
dx_{1,1} &  dx_{1,2}  & \ldots & dx_{1,r}
\end{vmatrix}
$$
is a weight vector for the $\GL(E)$ action with weight
$(2,1^{r-2},0^{n-r+1})$. If $u = |x_{ij} |$ for $1 \le i,j \le r$ then
the invariant form $u^{m-1}\delta$ is a weight vector for $(m+1,
m^{r-2}, m-1, 0^{n-r})$.

\begin{lemma} \label{kerbeta} As $\GL(E)$ representation $\Ker
d^3_m \simeq \Sch_{ (m^{r})} E^\ast \oplus \Sch_{(m^{r-1},
m-1,1)} E^\ast$ for $m \ge 2$ and $\Ker d^3_1 \simeq \wedge^r E^\ast$.
\end{lemma}
\begin{proof} From \eqref{vv} and \eqref{gg} we must show that the
endomorphism on $\Sch_{(m+1, m^{r-2}, m-1)} E^\ast$ induced by $d^3$ is
non-zero. To do this it is enough by Schur's Lemma to show that
$d^3(u^{m-1}\delta) \ne 0$, which by linearity is the same as
$d^3(\delta) \ne 0$. From \eqref{betaeq} we get for $X \in
\mathfrak{sl}_r$ that $d^3(\delta)(X)$ is the determinant
$$\begin{vmatrix}
x_{1,1} &  x_{1,2}  & \ldots & x_{1,r}\\
x_{2,1}   &  x_{2,2} & \ldots & x_{2,r}\\
\vdots & \vdots & \ddots & \vdots\\
x_{r-1,1}   &  x_{r-1,2} & \ldots & x_{r-1,r}\\
\rho_{1,1}(X) &  \rho_{1,2}(X)  & \ldots & \rho_{1,r}(X)
\end{vmatrix}
$$
so let $X = w_r^\ast \otimes w_1$. Then $\rho(X) = \sum_i x_{i,1}
\frac{\partial}{\partial x_{i,r}}$ and the last row in the determinant
is $(0, \dots ,0,x_{1,1})$. Thus $d^3(\delta)(X) \ne 0$.
\end{proof}

\begin{proposition} \label{h1mo} If $A$ is the Pl\"ucker algebra for $\GG(r,E)$ with
  $\dim E = n$, then
  $H^1_\mathfrak{m}(\Omega_A) = 0$.
\end{proposition}
\begin{proof} 
From Lemma \ref{kerd} and Lemma
\ref{kerbeta} we get  $\Ker d_{m+1}^3 =
\Image d_{m}^2$ for $m\ge 1$ and clearly $\Image d_0^2 \simeq \wedge^r
E^\ast \simeq \Ker d_{1}^3$. Thus $\Ker d^3 = \Image d^2$.
\end{proof}

\section{Cotangent cohomology for isotropic Grassmannians} \label{sec:iso}
In this section, we partially extend our vanishing results for Pl\"ucker algebras to the setting of isotropic Grassmannians.
Fix $n\geq 2$, $1\leq r \leq n$ and let $\LG(r,2n)$, $\OG(r,2n)$, and $\OG(r,2n+1)$ respectively denote the symplectic/orthogonal Grassmannians of isotropic $r$-planes in a $2n$ (or $2n+1$)-dimensional vector space. To avoid degenerate cases, and those coinciding with classical Grassmannians, we will make the following assumptions throughout:
\begin{enumerate}
\item For $\LG(r,2n)$, $r>1$ and $n\geq 2$;
\item For $\OG(r,2n)$, $n\geq 4$ and $r\neq n-1$;
\item For $\OG(r,2n+1)$, $r\geq 1$ and $n\geq 2$.
\end{enumerate}
Note that $\OG(n,2n)$ designates one of the two connected components of the Grassmannian of isotropic $n$ planes in a $2n$-dimensional vector space.
We consider each such Grassmannian in its Pl\"ucker embedding, and denoting its coordinate ring by $A$ and Serre's twisting sheaf by $\CO(1)$. Set $d=\dim A=\dim X+1$, where $X$ is the appropriate isotropic Grassmannian.
Our main result is 

\begin{theorem} \label{main2} Assume $A$ is the coordinate ring for an isotropic 
  Grassmannian $X$ different from $\LG(3,6)$. Then  $T^i_A = 0$
  for $2 \le i\leq d-3$, and 
  $T^{d-2}_A = 0$ if and only if  
 $X$ is either
$\LG(n-1,2n)$ or $\OG(n,2n+1)$.
Furthermore, $T^1_A = 0$ as long as $X$ is not an isotropic Grassmannian of $1$ or $2$-planes, or $\OG(4,8)$.
\end{theorem}
\begin{proof}
Combine Proposition \ref{svanes} with Theorems \ref{thm:main} and \ref{thm:main2} below.
\end{proof}

In addition to being useful for proving Theorem \ref{main2}, the following cohomology vanishing is interesting in its own right:
\begin{theorem}\label{thm:main} Let $X$ be $\LG(r,2n)$, $\OG(r,2n)$, or $\OG(r,2n+1)$.
The cohomology 
  $$H^i(X,\Theta_X(m))$$ vanishes for all $m\in \ZZ$ and $2\leq i \leq d-3$, except for $X=\LG(3,6)$.
The cohomology $$H^1(X,\Theta_X(m))$$ vanishes for all $m\in \ZZ$ if $r\neq 1,2$ and $X\neq \OG(4,8)$. Conversely, this cohomology group is non-zero for some $m\in \ZZ$ if $X$ is $\LG(2,2n)$ for $n\neq 3$, $\OG(1,2n)$, $\OG(4,4)$, or $\OG(1,2n+1)$.
Finally, the cohomology
 $$H^{d-2}(X,\Theta_X(m))$$ vanishes for all $m\in \ZZ$ if and only if $X$ is either
$\LG(n-1,2n)$ or $\OG(n,2n+1)$.
\end{theorem}
\begin{theorem}\label{thm:main2}
For $X=\LG(r,2n)$, $X=\OG(r,2n)$, or $X=\OG(r,2n+1)$, the cohomology 
  $H^i(X,\CO_X(m))$ vanishes for all $m\in \ZZ$ for all $1\leq i \leq d-2$.
\end{theorem}

Let $\CR$ be the tautological bundle on $X$, and $\CR^\vee$ the orthogonal complement. 
Then there are exact sequences 
$$
0 \to \CR^*\otimes (\CR^\vee/\CR) \to \Theta_X \to D_2(\CR^*)\to 0
$$
when $X$ is a symplectic Grassmannian, and 
$$
0 \to \CR^*\otimes (\CR^\vee/\CR) \to \Theta_X \to \bigwedge^2\CR^*\to 0
$$
when $X$ is an orthogonal Grassmannian, see \cite[Ch.~4 Ex.~9 \& 10]{we:coh}. Here $ D_2(\CR^*)=(\Sym^2 \CR)^*$ is the second divided power.
We will prove Theorem \ref{thm:main} by considering the long exact sequence of cohomology of twists of these short exact sequences. For this, we need the following vanishing results for the left and right terms in the above sequences:

\begin{lemma}\label{lemma:lga} 
The cohomology $$H^i(\LG(r,2n),D_2(\CR^*)(m))$$
   vanishes for all $m\in \ZZ$ and $2\leq i\leq d-3$, except for $\LG(3,6)$.
The cohomology $$H^1(\LG(r,2n),D_2(\CR^*)(m))$$ vanishes for all $m\in \ZZ$ if and only if $(r,n)\neq (2,2)$. 
Finally, the cohomology
$$H^{d-2}(\LG(r,2n),D_2(\CR^*)(m))$$  vanishes for all $m\in \ZZ$ if and only if $r\neq n$.
\end{lemma}

\begin{lemma}\label{lemma:oga} 
Let $X$ be $\OG(r,2n)$ or $\OG(r,2n+1)$.
The cohomology $$H^i(X,(\bigwedge^2 \CR^*)(m))$$
   vanishes for all $m\in \ZZ$ and $2\leq i \leq d-3$.
The cohomology $$H^1(X,(\bigwedge^2 \CR^*)(m))$$ vanishes for all $m\in \ZZ$ if and only if $X$ is not equal to $OG(1,2n)$, $OG(4,4)$, or $OG(1,2n+1)$. 
Finally, the cohomology
$$H^{d-2}(X,(\bigwedge^2 \CR^*)(m))$$
 vanishes for all $m\in \ZZ$ if and only if $X$ is not equal to $\OG(1,2n)$, $\OG(n,2n)$, or $\OG(1,2n+1)$.
\end{lemma}

\begin{lemma}\label{lemma:b} 
Let $X$ be $\LG(r,2n)$ or $\OG(r,2n)$ with $r<n$, or $\OG(r,2n+1)$.
The cohomology $$H^i(X,\CR^*\otimes (\CR^\vee/\CR)(m))$$
   vanishes for all $m\in \ZZ$ and $2\leq i \leq d-3$.
The cohomology $$H^1(X,\CR^*\otimes (\CR^\vee/\CR)(m))$$ vanishes for all $m\in \ZZ$ if and only if $X$ is not equal to $\LG(2,n)$ for $n>3$, $\OG(r,2n)$ for $r=1,2$, or $\OG(r,2n+1)$ for $r=1,2$ with $r\neq n$. 
Finally, the cohomology
$$H^{d-2}(X,\CR^*\otimes (\CR^\vee/\CR)(m))$$
 vanishes for all $m\in \ZZ$ if and only if $X$ equals $\LG(n-1,2n)$ or $\OG(n,2n+1)$.
\end{lemma}

\begin{proof}[Proof of Lemmata \ref{lemma:lga}, \ref{lemma:oga}, and \ref{lemma:b}]
We will prove these lemmata using Bott's theorem for isotropic Grassmannians \cite[4.3.4, 4.3.7, \& 4.3.9]{we:coh}. First we need some notation.
Let $\g$ be one of the Lie algebras $\spn_n$, $\so_{2n}$, or $\so_{2n+1}$, $\alpha_1,\ldots,\alpha_n$ its simple roots, and $\delta_1,\ldots,\delta_n$ the corresponding fundamental weights. We always assume that $n>1$. The positive roots of $\g$ are exactly as listed in Table \ref{table:roots}.

\begin{table}
\scriptsize{
\hrule

\vspace{.1cm}

{$\g=\spn_n$}

\vspace{.1cm}

\hrule
\begin{align*}
\alpha_i+\ldots+\alpha_j &\qquad i\leq j\leq n\\
2\alpha_j+\ldots+2\alpha_{n-1}+\alpha_n &\qquad  j <n\\
\alpha_i+\ldots+\alpha_{j-1}+2\alpha_j+\ldots+2\alpha_{n-1}+\alpha_n &\qquad i< j <n
\end{align*}
\hrule

\vspace{.1cm}

{$\g=\so_{2n}$}

\vspace{.1cm}

\hrule

\begin{align*}
\alpha_{j} &\qquad j=n-1,n\\
\alpha_i+\ldots+\alpha_{j} &\qquad i\leq j \leq n-2\\
\alpha_i+\ldots+\alpha_{n-2}+\alpha_{j} &\qquad i\leq n-2, j=n-1,n\\
\alpha_i+\ldots+\alpha_{n} &\qquad i\leq n-2 \\
\alpha_i+\ldots+\alpha_{j-1}+2\alpha_j+\ldots +2\alpha_{n-2}+\alpha_{n-1}+\alpha_n &\qquad i< j\leq n-2
\end{align*}
\hrule

\vspace{.1cm}

{$\g=\so_{2n+1}$}

\vspace{.1cm}

\hrule
\begin{align*}
\alpha_i+\ldots+\alpha_j &\qquad i\leq j\leq n\\
\alpha_i+\ldots+\alpha_{j-1}+2\alpha_j+\ldots +2\alpha_{n} &\qquad i<j\leq n
\end{align*}

}
\caption{Positive roots of $\g$}\label{table:roots}
\end{table}

For any weight $\beta$ of type $A_{r-1}$ let $K_\beta$ be the corresponding Weyl functor, and for any weight $\mu$ of type $B_{n-r}$, $C_{n-r}$, or $D_{n-r}$ let $V_\mu(\CR^\vee/\CR)$ be the bundle defined fiberwise by the representation of weight $\mu$ with respect to the symplectic/orthogonal fibers of $\CR^\vee/\CR$. We then have 
\begin{align*}
D_2(\CR^*)=K_\beta(\CR)\qquad &\textrm{for}\qquad \beta=(0,\ldots,0,-2)\\
\bigwedge^2\CR^*=K_\beta(\CR)\qquad &\textrm{for}\qquad \beta=(0,\ldots,0,-1,-1)\\
\CR^*\otimes (\CR^\vee/\CR)=K_\beta(\CR)\otimes V_\mu(\CR^\vee/\CR)\qquad &\textrm{for}\qquad \beta=(0,\ldots,0,-1),\ \mu=(1,0,\ldots,0).
\end{align*}
By Bott's theorem \cite[4.3.4, 4.3.7, \& 4.3.9]{we:coh}, 
 the $i$th cohomology of the twist by $\CO(m)$ of the above bundles is non-zero exactly when the weight $\gamma$ is non-singular of index $i$, where $\gamma$ is respectively
\begin{align*}
&\gamma=2\delta_1+m\delta_r+\sum_{j=1}^n \delta_j;\\
&\gamma=\delta_1+\delta_2+m\delta_r+\sum_{j=1}^n \delta_j; \textrm{ or}\\
&\gamma=\delta_1+m\delta_r+\delta_{r+1}+\sum_{j=1}^n \delta_j.
\end{align*}

Recall that the index of $\gamma$ is the number of positive roots $\alpha$ such that $\alpha(\gamma)<0$. Note that the only roots for which this can occur are those which involve $\alpha_r$, of which there are exactly $d-1$. Denote this set of $d-1$ roots by $S$.
Furthermore, if $\alpha(\gamma)<0$ for any positive root $\alpha$, then $\alpha_r(\gamma)<0$. In Tables \ref{table:val1}, \ref{table:val4}, and \ref{table:val3}, we list all values of $\alpha(\gamma)$ for those $\alpha\in S$ with $\alpha(\delta_r)=1$, the maximal value of $\alpha(\gamma)$ for those $\alpha\in S$ with $\alpha(\delta_r)=2$, along with (in some cases) further values of $\alpha(\gamma)$.
These lists follow from Table \ref{table:roots} by inspection.

\begin{table}
{\scriptsize
\begin{tabular}{l l@{\qquad} l@{\qquad} l l}
\hline
$\g$ &$r,n$ & $\alpha(\gamma)$ for  $\# \alpha_r=1$ & Max $\alpha(\gamma)$ for $\#\alpha_r=2$ & Other $\alpha(\gamma)$\\
\hline
\\
$\spn_n$ & $1<r <n-1$ & $m+1,\ldots,m+n+(n-r)+1$ & $2m+2n+3$ & \\
$\spn_n$ & $1<r =n-1$ & $m+1,\ldots,m+n-1,m+n+1,m+n+2$ & $2m+2n+3$ & $2(m+n)$\\
$\spn_n$ & $3<r=n$ & $m+1,\ldots,m+2n,m+2n+3$ && \\
$\spn_n$ & $3=r=n$ & $m+1,m+2,m+3,m+5,m+6,m+9$ && \\
$\spn_n$ & $2=r=n$ & $m+1,m+4,m+7$ && \\
\\
\hline
\\
\end{tabular}
}
\caption{$\alpha(\gamma)$ for $\gamma=2\delta_1+m\delta_r+\sum \delta_j$}\label{table:val1}
\end{table}

\begin{table}
{\scriptsize
\begin{tabular}{l l@{\qquad} l@{\qquad} l l}
\hline
$\g$ &$r,n$ & $\alpha(\gamma)$ for  $\# \alpha_r=1$ & Max $\alpha(\gamma)$ for $\#\alpha_r=2$ & Other $\alpha(\gamma)$\\
\hline
\\
$\so_{2n}$ & $2<r <n-1$ & $m+1,\ldots,m+n+(n-r)$ & $2(m+n)$ & \\
$\so_{2n}$ & $r\geq n-1,n>4$ & $m+1,\ldots,m+2n-2,m+2n$ &  & \\
$\so_{2n}$ & $r=2,n\geq 4$ & $m+2,\ldots,m+2n-2$ &$2(m+n)$  & \\
$\so_{2n}$ & $r=1,n\geq 4$ & $m+2,m+4,\ldots,m+2n-2,m+2n$ &  & \\
\\
$\so_{2n+1}$ & $2<r <n$ & $m+1,\ldots,m+n+(n-r)+2$ & $2(m+n+1)$ & \\
$\so_{2n+1}$ & $2=r <n$ & $m+2,\ldots,m+2n$ & $2(m+n+1)$ & \\
$\so_{2n+1}$ & $2<r=n$ & $m+1\ldots,m+n-2,m+n,m+n+2$ &$2(m+n+1)$&$2(m+n-1)$ \\
$\so_{2n+1}$ & $r=2,n=2$ & $m+2,m+4$ & $2(m+3)$& \\
$\so_{2n+1}$ & $r=1$ & $m+2,m+4,\ldots,m+2n,m+2n+2$ && \\
\\
\hline
\\
\end{tabular}
}
\caption{$\alpha(\gamma)$ for $\gamma=\delta_1+\delta_2+m\delta_r+\sum \delta_j$}\label{table:val4}
\end{table}

\begin{table}
{\scriptsize
\begin{tabular}{l l@{\qquad} l@{\qquad} l l}
\hline
$\g$ &$r,n$ & $\alpha(\gamma)$ for  $\# \alpha_r=1$ & Max $\alpha(\gamma)$ for $\#\alpha_r=2$ & Other $\alpha(\gamma)$\\
\hline
\\
$\spn_n$ & $2<r <n-1$ & $m+1,\ldots,m+n+(n-r),m+n+(n-r)+2$ & $2m+2n+3$ & \\
$\spn_n$ & $2<r =n-1$ & $m+1,\ldots,m+n,m+n+2$ & $2(m+n+1)$ & \\
$\spn_n$ & $2=r<n-1$ & $m+1,m+3,\ldots,m+2n-2,m+2n$ &$2m+2n+3$& \\
$\spn_n$ & $r=2,n=3$ & $m+1,m+3,m+5$ & $2(m+4)$&$2(m+2)$ \\
\\
$\so_{2n}$ & $2<r <n-1$ & $m+1,\ldots,m+n+(n-r)-1,m+n+(n-r)+1$ & $2(m+n)$ & \\
$\so_{2n}$ & $r=2,n\geq 4$ & $m+1,m+3,\ldots,m+2n-3,m+2n-1$ &$2(m+n)$  & \\
$\so_{2n}$ & $r=1,n\geq 4$ & $m+2,m+4,\ldots,m+2n-2,m+2n$ &  & \\
\\
$\so_{2n+1}$ & $2<r <n-1$ & $m+1,\ldots,m+n+(n-r)+1,m+n+(n-r)+3$ & $2(m+n+1)$ & \\
$\so_{2n+1}$ & $2<r =n-1$ & $m+1,\ldots,m+n+2,m+n+4$ & $2(m+n+1)$ & \\
$\so_{2n+1}$ & $2<r =n$ & $m+1,\ldots,m+n-1,m+n+1$ & $2m+2n$ & \\
$\so_{2n+1}$ & $2=r<n-1$ & $m+1,m+3,\ldots,m+2n-1,m+2n+1$ &$2(m+n+1)$& \\
$\so_{2n+1}$ & $r=2,n=3$ & $m+1,m+3,m+5,m+7$ & $2(m+4)$& \\
$\so_{2n+1}$ & $r=2,n=2$ & $m+1,m+3$ & $2(m+2)$& \\
$\so_{2n+1}$ & $r=1$ & $m+2,m+4,\ldots,m+2n,m+2n+2$ && \\
\\
\hline
\\
\end{tabular}
}
\caption{$\alpha(\gamma)$ for $\gamma=\delta_1+m\delta_r+\delta_{r+1}+\sum \delta_j$}\label{table:val3}
\end{table}

The claims of the lemmata now follow from Tables \ref{table:val1}, \ref{table:val4}, and \ref{table:val3}. Indeed, suppose that $\gamma$ is non-singular of some strictly positive index $i$. Then the values of $\gamma(\alpha)$ cannot contain $0$, must contain $i$ negative values, and must contain $d-i-1$ positive values. Inspection of the tables leads to bounds on $i$. For example, consider the case $\g=\spn_n$, $1<k=n-1$, and $\gamma=2\delta_1+m\delta_k+\sum\delta_j$ (see Table \ref{table:val1}). It follows that $m+1<0$, from which follows that $m+n-1<0$. If $m+n=0$, then $2(m+n)=0$ as well, which is impossible, since $\gamma$ is non-singular. So in fact, $m+n<0$, as are also $m+n+1$ and $m+n+2$. We thus conclude that in this case, $i=d-1$. All other cases are similarly straightforward.
\end{proof}

\begin{proof}[Proof of Theorem \ref{thm:main}]
  If $r=n$, then $\CR^\vee/\CR=0$, so $\Theta_X=D_2(\CR^*)$ or $\Theta_X=\bigwedge^2 \CR^*$ and the claims follow directly from Lemmata \ref{lemma:lga} and \ref{lemma:oga}. For $r<n$, we apply Lemmata \ref{lemma:lga}, \ref{lemma:oga} and \ref{lemma:b} to the long exact sequence of cohomology. For the claim regarding $H^1$ for $LG(2,2n)$ with $n>3$, note that 
$H^1(LG(2,2n),D_2(\CR^*)(-2))$ is non-vanishing, but $H^0(LG(2,2n),\CR^*\otimes(\CR^\vee/\CR))(-2))=0$.
\end{proof}

\begin{proof}[Proof of Theorem \ref{thm:main2}]
By Bott's theorem, the  $i$th cohomology of $\CO_X(m)$ vanishes unless the weight
$$
\gamma=m\delta_r+\sum_{j=1}^n\delta_j
$$
is non-singular of index $i$. The claim now follows from arguments similar to those used to prove the above lemmata.
\end{proof}

\section{Deforming complete intersections in cones over 
\\ Grassmannians}
\begin{lemma} \label{cilemma} Let $A$ be a $d$-dimensional $k$-algebra
with $T^i_A=0$ for $1 \le i\le d$. If $I$ is a complete intersection
ideal in $A$ then $T_A^1(A/I)=0$.
\end{lemma}
\begin{proof} Let $B=A/I$.  We have a long exact sequence
$$
\cdots \to T^i_A(I) \to T^i_A \to T^i_A(B) \to T^{i+1}_A(I) \to \cdots
$$
Let $F$ be a free $A$-module of rank equal to the number of generators
of $I$ and consider the resolution of $I$ by the Koszul complex
$$0 \to \bigwedge^l F\xrightarrow{d_l} \dotsm \xrightarrow{d_3}  \bigwedge^2 F
\xrightarrow{d_2} F \xrightarrow{d_1} I \to 0
$$
 which we can split into short exact sequences
$$0 \to I_{j} \to \bigwedge^jF \to I_{j-1} \to 0
$$
with $I_0:=I$ and $I_j:=\ker d_j$.

We show that $T_A^p(I_j)=0$ for $j+2\geq p>1$ by induction on
$j$. Indeed, $T_A^p(I_l)=0$ for all $p>1$ since $I_l=0$.  Suppose that
we have shown $T_A^p(I_{j})=0$ for all $j+2\geq p>1$.  Consider any
$p$ satisfying $j+1\geq p>1$. Then $T_A^p(I_{j-1})$ vanishes if
$T_A^p(\bigwedge^j F)$ does. But $\bigwedge^j F$ is free, so $T_A^p=
0$ implies $T_A^p(\bigwedge^j F) = 0$. Thus since both $T_A^1$ and
$T_A^2(I)$ vanish, we get $T_A^1(B)=0$ as desired.
\end{proof}

Let $X = \Proj A \subseteq \PP^n$. We say that $Y \subset X$ is a
complete intersection in $X$ if $Y = \Proj B$ is of codimension $l$ in
$X$ with $B = A/(f_1, \dots , f_l)$ for $l$ homogeneous polynomials in
$k[x_0, \dots , x_n]$.
\begin{proposition} Let $X = \Proj A \subseteq \PP^n$ have dimension
$d-1$ and assume $Y = \Proj B \subset X$ is a complete intersection in
$X$. Let $\mathfrak{m}$ be the irrelevant maximal ideal in $k[x_0,
\dots , x_n]$. If
\begin{list}{\textup{(\roman{temp})}}{\usecounter{temp}}
\item $\depth_\mathfrak{m} B \ge 3$
\item $H^2(Y,\CO_Y) = 0$
\item $T^i_A=0$ \text{ for $1 \le i\le d$}
\end{list} then any deformation of $Y$ is again a complete
intersection in $X$.
\end{proposition}
\begin{proof} The statement will follow if the forgetful map
$\Def_{Y/X} \to \Def_Y$ from the local Hilbert functor of $Y$ in $X$ to
the deformation functor of $Y$ is smooth. This follows if
$T^1_X(\CO_Y) = 0$. Combine Lemma 4.2, Lemma 4.3 and Proposition 4.21
in \cite{ck:com} to see that the first two conditions guarantee a
surjection $T^1_A(B)_0 \to T^1_X(\CO_Y)$. Thus the third assumption
and Lemma \ref{cilemma} imply the result.
\end{proof}

\begin{corollary} Let $A$ be the Pl\"ucker algebra for $\GG(r,n)$,
$d=\dim A = n(n-r)+1$ and $X=\Proj A[x_1,...,x_m]$. If $Y$ is a
complete intersection of codimension less than $d$ in $X$ then any
deformation of $Y$ is again a complete intersection in $X$.
\end{corollary}

\begin{remark}Let $X$ be as above, and let $Y$ be a complete intersection of type
$(a_1,\ldots, a_k)$ in $X$, where $m\leq k<d$, and $\sum a_i<n$. Then $Y$ is a
(possibly singular) Fano variety. By the above corollary, any
smoothing of $Y$ is again a complete intersection of type
$(a_1,\ldots,a_k)$ in $X$.
\end{remark}

\bibliographystyle{amsalpha} 

\begin{thebibliography}{RWW14}

\bibitem[And74]{an:hom}
Michel Andr{\'{e}}, \emph{Homologie des alg{\`{e}}bres commutatives},
  Springer-Verlag, 1974.

\bibitem[BC91]{bc:hyp}
Kurt Behnke and Jan~Arthur Christophersen, \emph{Hypersurface sections and
  obstructions (rational surface singularities)}, Compos. Math. \textbf{77}
  (1991), 233--268.

\bibitem[CK14]{ck:com}
Jan~Arthur Christophersen and Jan~O. Kleppe, \emph{Comparison theorems for
  deformation functors via invariant theory}, arXiv:1209.3444v3 [math.AG],
  2014.

\bibitem[FH91]{fh:rep}
William Fulton and Joe Harris, \emph{Representation theory}, Graduate Texts in
  Mathematics, vol. 129, Springer-Verlag, 1991.

\bibitem[Gro05]{gr:sga2}
Alexander Grothendieck, \emph{Cohomologie locale des faisceaux coh{\'e}rents et
  th{\'e}or{\`e}mes de {L}efschetz locaux et globaux ({SGA} 2)}, Documents
  Math{\'e}matiques (Paris) [Mathematical Documents (Paris)], 4,
  Soci{\'e}t{\'e} Math{\'e}matique de France, Paris, 2005.

\bibitem[Ill71]{il:com}
Luc Illusie, \emph{Complexe cotangent et d{\'e}formations. {I}}, Lecture Notes
  in Mathematics, Vol. 239, Springer-Verlag, Berlin, 1971.

\bibitem[Lau79]{la:for}
Olav~Arnfinn Laudal, \emph{Formal moduli of algebraic structures}, Lecture
  Notes in Mathematics, vol. 754, Springer-Verlag, 1979.

\bibitem[RWW14]{rww:loc}
Claudiu Raicu, Jerzy Weyman, and Emily~E. Witt, \emph{Local cohomology with
  support in ideals of maximal minors and sub-maximal {P}faffians}, Adv. Math.
  \textbf{250} (2014), 596--610.

\bibitem[Sch71]{sc:rig}
Michael Schlessinger, \emph{Rigidity of quotient singularities}, Invent. Math.
  \textbf{14} (1971), 17--26.

\bibitem[Sva75]{sv:som}
Torgny Svanes, \emph{Some criteria for rigidity of noetherian rings}, Math. Z.
  \textbf{144} (1975), 135--145.

\bibitem[Wah97]{wa:coh}
Jonathan Wahl, \emph{On cohomology of the square of an ideal sheaf}, J.
  Algebraic Geom. \textbf{6} (1997), 481--511.

\bibitem[Wey03]{we:coh}
Jerzy Weyman, \emph{Cohomology of vector bundles and syzygies}, Cambridge
  University Press, 2003.

\end{thebibliography}

\providecommand{\bysame}{\leavevmode\hbox to3em{\hrulefill}\thinspace}
\providecommand{\MR}{\relax\ifhmode\unskip\space\fi MR }
\providecommand{\MRhref}[2]{%
  \href{http://www.ams.org/mathscinet-getitem?mr=#1}{#2}
}
\providecommand{\href}[2]{#2}

\end{document}